\pdfoutput=1
\documentclass[11pt]{amsart}

\usepackage[paper,thm/numberwithin={},pkg/amsaddr=true]{zbs}
\setlist[enumerate,1]{ref=(\arabic*)}

\title[Correction to Chinburg-Henriksen result on powers of integer polynomials]{A correction to a result of Chinburg and Henriksen on powers of integer polynomials}
\author{Daniel G. Zhu}
\address{Department of Mathematics, Princeton University, Princeton, NJ 08544, USA}
\email{zhd@princeton.edu}

\begin{document}
\begin{abstract}
For a positive integer $k$, let $m(k)$ be the minimum positive integer $m$ such that $mx$ can be written as an integer linear combination of $k$th powers of integer polynomials. We correct an error in a 1976 formula of Chinburg and Henriksen for $m(k)$.
\end{abstract}
\maketitle

Given a positive integer $k$ and a ring $R$, let $J(k, R)$ be the additive subgroup of $R$ generated by $k$th powers of elements of $R$ and let $G(k, R) \subseteq \setz$ consist of the integers $m$ for which $mR \subseteq J(k, R)$. Since $G(k, R)$ is an additive subgroup of $\setz$, it is $m(k, R)\setz$ for some unique $m(k, R)\geq 0$. Furthermore, let $m(k) = m(k, \setz[x])$, a definition motivated by the fact that for any homomorphism $\varphi\colon\setz[x]\to R$ and $m \in G(k, R)$, we have $m\varphi(\setz[x]) \subseteq \varphi(J(k, \setz[x])) \subseteq J(k, R)$. By varying $\varphi$, we find that $mR \subseteq J(k, R)$ for all $m \in G(k, \setz[x])$, which implies that $G(k, \setz[x]) \subseteq G(k, R)$ and thus $m(k, R) \mid m(k)$. A similar argument also shows that $m(k)$ can be defined as the minimum positive integer $m$ such that $mx \in J(k, \setz[x])$ (or $0$ if no such integer exists, though we will rule out this possibility below).

In 1976, Chinburg and Henriksen \cite{ch76} published an explicit formula for $m(k)$. To state their result, we first note two basic properties of $m(k)$:
\begin{itemize}
\item Since the $x$-coefficient of any $k$th power in $\setz[x]$ is always divisible by $k$, we must have $k \mid m(k)$.
\item By the theory of finite differences, the polynomial
\[\sum_{i=0}^{k-1} (-1)^{i} \binom{k-1}{i} (x - i)^k\]
is of the form $k!x + a$ for some $a \in\setz$, implying that $k!x \in J(k, \setz[x])$ and thus that $m(k) \mid k!$. (We may ignore the constant term as $1 \in J(k, \setz[x])$.)
\end{itemize}
As a result, we may decompose
\[m(k) = k\cdot \prod_{p\mid k} p^{\alpha_k(p)} \cdot \prod_{\substack{p < k \\ p \mkern1.5mu\nmid\mkern1.5mu k}} p^{\beta_k(p)}\]
for some nonnegative integers $\alpha_k(p)$ and $\beta_k(p)$, where the product is taken over primes $p$. Chinburg and Henriksen incorrectly claimed the following:
\begin{enumerate}
\item\label{item:p1} For $p$ odd and $k$ a multiple of $p$, we have $\alpha_k(p) = 1$ if $k > p$ and $\alpha_k(p) = 0$ if $k = p$.
\item\label{item:p2} For $k$ even, we have $\alpha_k(2) = 0$ if $k = 2$, $\alpha_k(2) = 2$ if $k$ is divisible for $2^j - 1$ for some positive integer $j \geq 2$, and $\alpha_k(2) = 1$ for all other $k$.
\item\label{item:p3} For $k$ not a multiple of $p$, we have $\beta_k(p) = 1$ if $k$ is divisible by a positive integer of the form $(p^{mr}-1)/(p^r-1)$ for positive integers $m \geq 2$ and $r$, and $\beta_k(p) = 0$ otherwise.
\end{enumerate}
While items \ref{item:p1} and \ref{item:p3} are correct, item \ref{item:p2} is not. The corresponding error in \cite{ch76} appears at the end of the proof of Proposition 13, where an identity $2^{n+1} x = \sum_{i=1}^q a_ig_i(x)^k$ modulo an ideal $I \subseteq \setz[x]$ is specialized to $0 = \sum_{i=1}^q a_ig_i(0)^k$ and $2^{n+1} = \sum_{i=1}^q a_ig_i(1)^k$ modulo the same ideal $I$: an invalid manipulation, as plugging in $x = 0$ and $x = 1$ to an element of $I$ does not necessarily yield an element of $I$. The purpose of this note is to prove the following corrected version of item \ref{item:p2}:
\begin{thm} \label{thm:main}
We have $\alpha_2(2) = 0$. For $k > 2$ even, we have $\alpha_k(2) = 2$ if $k$ is divisible by $6$ and $\alpha_k(2) = 1$ otherwise.
\end{thm}
The first case where this differs from the Chinberg-Henricksen formula is $k = 14$. Instead of $14 \cdot 2^2 \cdot 7 \cdot 13 = 5096$, we have $m(14) = 14 \cdot 2 \cdot 7 \cdot 13 = 2548$. More numerical data can be found in the \hyperref[sec:appendix]{appendix}, as well as sequences \href{https://oeis.org/A370252}{A370252}, \href{https://oeis.org/A005729}{A005729}, \href{https://oeis.org/A005730}{A005730}, and \href{https://oeis.org/A005731}{A005731} in \cite{oeis}.

\subsection*{Preliminaries} Let $v_p$ denote $p$-adic valuation and let $n = v_2(k)$. To begin, we recall some results from \cite{ch76}.

\begin{lem}[Special case of {\cite[Prop.\ 4(b)]{ch76}}] \label{lem:ak2}
We have $\alpha_2(2) = 0$. If $k > 2$ is even, then $v_2(m(k)) = v_2(m(k, \setz/2^{n+2}\setz[x]))$. In particular, $\alpha_k(2) \leq 2$.
\end{lem}
\begin{lem}[{\cite[Prop.\ 11]{ch76}}] \label{lem:ak1}
For a prime $p$ and positive integer $k$ with $p \mid k$ and $p < k$, we have $\alpha_k(p) \geq 1$.
\end{lem}
\begin{lem}[Special case of {\cite[Prop.\ 3]{ch76}}] \label{lem:vpbinom}
For positive integers $k\geq j$ with $k$ even, $v_2(\binom{k}{j}) \geq v_2(k) - (j-1)$, with equality holding if and only if $j \in \set{1,2}$.
\end{lem}
\begin{lem}[{\cite[Lem.\ 8(c)]{ch76}}] \label{lem:pid}
For every prime $p$ and positive integer $k$, there is a nonzero ideal contained in $J(k, \setf_p[x])$. 
\end{lem}

\subsection*{Proof of \texorpdfstring{\cref{thm:main}}{Theorem \ref{thm:main}} for \texorpdfstring{$k$}{k} a multiple of \texorpdfstring{$6$}{6}} In light of \cref{lem:ak2} we need to show that $\alpha_k(2) \geq 2$. Since $m(k, R) \mid m(k)$ for any ring $R$ and $m(k', R) \mid m(k, R)$ for any factor $k'$ of $k$, it suffices to show the following claim:
\begin{prop}
For any nonnegative integer $s$, we have
\[J(6 \cdot 2^s, \setz[x]/(x^2+x+1)) \subseteq \setmid{a + 2^{s+3} bx}{a,b \in \setz}.\]
In particular, $2^{s+3} \mid m(6 \cdot 2^s, \setz[x]/(x^2+x+1))$.
\end{prop}
\begin{proof}
We use induction on $s$. To handle $s = 0$, we compute (in $\setz[x]/(x^2+x+1)$)
\[(a + bx)^6 = (a^6 - 15a^4b^2 + 20a^3b^3 - 6ab^5 + b^6) + (6a^5b - 15a^4b^2 + 15a^2b^4 - 6ab^5)x,\]
so we must show that $8$ divides
\[6a^5b - 15a^4b^2 + 15a^2b^4 - 6ab^5 = 3(a^2-b^2)((a^2-2ab+b^2)^2-(a^2-3ab+b^2)^2).\]
If $a$ and $b$ are both even, then this is clearly divisible by $8$. If $a$ and $b$ are both odd, then $a^2-b^2$ is a difference of two odd squares and is thus divisible by $8$. If $a$ and $b$ have different parities, then $(a^2-2ab+b^2)^2-(a^2-3ab+b^2)^2$ is a difference of two odd squares and is divisible by $8$.

For the inductive step, take some $s \geq 0$ and suppose that every $(6 \cdot 2^s)$-th power in $\setz[x]/(x^2+x+1)$ can be written as $a + 2^{s+3} bx$ for some $a,b\in\setz$. Then any $(6 \cdot 2^{s+1})$-th power can be written as
\[(a + 2^{s+3} bx)^2 = (a^2 - 2^{2s+6} b^2) + 2^{s+4}(ab - 2^{s+2}b^2)x,\]
which is of the form $a + 2^{s+4} bx$, as desired.
\end{proof}
\subsection*{Proof of \texorpdfstring{\cref{thm:main}}{Theorem \ref{thm:main}} for \texorpdfstring{$k$}{k} not a multiple of \texorpdfstring{$6$}{6}} If $k = 2$ we are done by \cref{lem:ak2}. Otherwise, by \cref{lem:ak1} we need to show that $\alpha_k(2) \leq 1$. By \cref{lem:ak2} it suffices to show that $m(k, \setz/2^{n+2}\setz[x]) \mid 2^{n+1}$, where $n = v_2(k)$.

For a ring $R$, let $K(k, R)$ be the additive subgroup generated by elements of the form $g^k$ and $g^k(h+h^2)$ for elements $g, h \in R$. We first reduce the problem to one over $\setf_2$.
\begin{prop}
$2^{n+1} K(k, \setf_2[x]) \subseteq J(k, \setz/2^{n+2}\setz[x])$. In particular, if $K(k, \setf_2[x]) = \setf_2[x]$, then $m(k, \setz/2^{n+2}\setz[x]) \mid 2^{n+1}$.
\end{prop}
\begin{proof}
After lifting elements of $\setf_2[x]$ to $\setz/2^{n+2}\setz[x]$, it suffices to show that for all $g, h \in \setz/2^{n+2}\setz[x]$, the elements $2^{n+1} g^k$ and $2^{n+1}g^k(h+h^2)$ are in $J(k, \setz/2^{n+2}\setz[x])$. This is obvious in the former case. In the latter case, note that by \cref{lem:vpbinom}, we have
\[g^k(1 + 2h)^k - g^k \equiv 2^{n+1}g^k(h+h^2) \pmod{2^{n+2}}.\qedhere\]
\end{proof}

We now need to show that $K(k, \setf_2[x]) = \setf_2[x]$.
\begin{prop} \label{prop:ff}
For any positive integer $j$, we have $K(k, \setf_{2^j}) = \setf_{2^j}$.
\end{prop}
\begin{proof}
Recall that an element of $\setf_{2^j}$ can be written in the form $h + h^2$ if and only if its trace (over $\setf_2$) is $0$. If $j$ is odd, then the trace of $1$ is $1$, so any element in $\setf_{2^j}$ is either of the form $h + h^2$ or $h + h^2 + 1$, each of which is clearly in $K(k, \setf_{2^j})$.

If $j$ is even, then since $k$ is not divisible by $3$, the three cube roots of $1$ (call them $1$, $\omega$, and $\omega^2$) in $\setf_{2^j}$ are all $k$th powers. For any element $r \in \setf_{2^j}$, the elements $r$, $\omega r$, and $\omega^2 r$ sum to $0$, so they cannot all have trace $1$. It follows that we can write $r = \omega^i(h + h^2)$ for some $i \in \set{0,1,2}$ and $h \in \setf_{2^j}$.
\end{proof}
The following argument concludes the proof:
\begin{prop}
$K(k, \setf_2[x]) = \setf_2[x]$.
\end{prop}
\begin{proof}
If we let $I$ be the ideal from \cref{lem:pid}, then $I \subseteq J(k, \setf_2[x]) \subseteq K(k, \setf_2[x])$, so we will be done if we can show that $K(k, \setf_2[x]/I) = \setf_2[x]/I$. This follows rather quickly from \cref{prop:ff} if $I$ is a radical ideal; in that case, since $\setf_2[x]$ is a principal ideal domain we may factor $I = (p_1)(p_2) \cdots (p_\ell)$ for distinct irreducible polynomials $p_1,p_2,\ldots,p_\ell \in \setf_2[x]$, implying that $\setf_2[x]/I$ is a product of finite fields, after which we are done by \cref{prop:ff} since it easy to show that $K(k, R_1 \times R_2) = K(k, R_1) \times K(k, R_2)$ for rings $R_1$ and $R_2$.

To tackle the general case, we claim that for any $f \in \setf_2[x]$ and ideal $I \subseteq \setf_2[x]$, we have that $f \in K(k, \setf_2[x]/I)$ implies $f \in K(k, \setf_2[x]/I^2)$. Indeed, if we have
\[f = \sum_{i=1}^s a_i g_i^k + \sum_{i=1}^t b_i (g'_i)^k (h_i + h_i^2) + r\]
for some integers $a_i, b_i \in \setz$ and elements $g_i, g'_i, h_i, r \in \setf_2[x]$ with $r \in I$, then by applying the Frobenius endomorphism we find
\[f^2 = \sum_{i=1}^s a_i (g_i^2)^k + \sum_{i=1}^t b_i ((g'_i)^2)^k (h_i^2 + (h_i^2)^2) + r^2,\]
so $f^2 \in K(k, \setf_2[x]/I^2)$. But $f + f^2 \in K(k, \setf_2[x]/I^2)$, so $f \in K(k, \setf_2[x]/I^2)$ as well.

The fact that $K(k, \setf_2[x]/I) = \setf_2[x]/I$ for all nonzero ideals $I$ follows from iterating the above process and the fact that every nonzero ideal $I$ divides a power of some nonzero radical ideal.
\end{proof}

\section*{Appendix} \label{sec:appendix}
We exhibit a corrected version of the table in \cite{ch76}, with bolded rows signifying corrections due to \cref{thm:main}. Two typos (at $k = 100, 128$) in the original have also been corrected. Here we let
\[a(k) = \prod_{p\mid k} p^{\alpha_k(p)} \quad\text{and}\quad b(k) = \prod_{\substack{p < k \\ p \mkern1.5mu\nmid\mkern1.5mu k}} p^{\beta_k(p)}.\]
The sequences $a(k)$, $b(k)$, $m(k)/k$, and $m(k)$ appear in \cite{oeis} as sequences \href{https://oeis.org/A005730}{A005730}, \href{https://oeis.org/A005731}{A005731}, \href{https://oeis.org/A005729}{A005729}, and \href{https://oeis.org/A370252}{A370252}, respectively.
\begin{center}
{\tiny
\begin{tabular}{ccccc} \toprule
$k$&$a(k)$&$b(k)$&$m(k)/k$&$m(k)$\\ \midrule
$1$&$1$&$1$&$1$&$1$\\
$2$&$1$&$1$&$1$&$2$\\
$3$&$1$&$2$&$2$&$6$\\
$4$&$2$&$3$&$6$&$24$\\
$5$&$1$&$2$&$2$&$10$\\
$6$&$4\cdot3=12$&$5$&$60$&$360$\\
$7$&$1$&$2$&$2$&$14$\\
$8$&$2$&$3\cdot7=21$&$42$&$336$\\
$9$&$3$&$2$&$6$&$54$\\
$10$&$2\cdot5=10$&$3$&$30$&$300$\\
$11$&$1$&$1$&$1$&$11$\\
$12$&$4\cdot3=12$&$5\cdot11=55$&$660$&$7920$\\
$13$&$1$&$3$&$3$&$39$\\
\boldmath$14$&\boldmath$2\cdot7=14$&\boldmath$13$&\boldmath$182$&\boldmath$2548$\\
$15$&$3\cdot5=15$&$2$&$30$&$450$\\
$16$&$2$&$3\cdot7=21$&$42$&$672$\\
$17$&$1$&$2$&$2$&$34$\\
$18$&$4\cdot3=12$&$5\cdot17=85$&$1020$&$18360$\\
$19$&$1$&$1$&$1$&$19$\\
$20$&$2\cdot5=10$&$3\cdot19=57$&$570$&$11400$\\
$21$&$3\cdot7=21$&$2$&$42$&$882$\\
$22$&$2\cdot11=22$&$1$&$22$&$484$\\
$23$&$1$&$1$&$1$&$23$\\
$24$&$4\cdot3=12$&$5\cdot7\cdot11\cdot23=8855$&$106260$&$2550240$\\
$25$&$5$&$2$&$10$&$250$\\
$26$&$2\cdot13=26$&$3\cdot5=15$&$390$&$10140$\\
$27$&$3$&$2$&$6$&$162$\\
\boldmath$28$&\boldmath$2\cdot7=14$&\boldmath$3\cdot13=39$&\boldmath$546$&\boldmath$15288$\\
$29$&$1$&$1$&$1$&$29$\\
$30$&$4\cdot3\cdot5=60$&$29$&$1740$&$52200$\\
$31$&$1$&$2\cdot5=10$&$10$&$310$\\
$32$&$2$&$3\cdot7\cdot31=651$&$1302$&$41664$\\
$33$&$3\cdot11=33$&$2$&$66$&$2178$\\
$34$&$2\cdot17=34$&$1$&$34$&$1156$\\
$35$&$5\cdot7=35$&$2$&$70$&$2450$\\
$36$&$4\cdot3=12$&$5\cdot11\cdot17=935$&$11220$&$403920$\\
$37$&$1$&$1$&$1$&$37$\\
$38$&$2\cdot19=38$&$37$&$1406$&$53428$\\
$39$&$3\cdot13=39$&$2$&$78$&$3042$\\
$40$&$2\cdot5=10$&$3\cdot7\cdot19=399$&$3990$&$159600$\\
$41$&$1$&$1$&$1$&$41$\\
$42$&$4\cdot3\cdot7=84$&$5\cdot13\cdot41=2665$&$223860$&$9402120$\\
$43$&$1$&$1$&$1$&$43$\\
$44$&$2\cdot11=22$&$3\cdot43=129$&$2838$&$124872$\\
$45$&$3\cdot5=15$&$2$&$30$&$1350$\\
$46$&$2\cdot23=46$&$1$&$46$&$2116$\\
$47$&$1$&$1$&$1$&$47$\\
$48$&$4\cdot3=12$&$5\cdot7\cdot11\cdot23\cdot47=416185$&$4994220$&$239722560$\\
$49$&$7$&$2$&$14$&$686$\\
$50$&$2\cdot5=10$&$3\cdot7=21$&$210$&$10500$\\
$51$&$3\cdot17=51$&$2$&$102$&$5202$\\
$52$&$2\cdot13=26$&$3\cdot5=15$&$390$&$20280$\\
$53$&$1$&$1$&$1$&$53$\\
$54$&$4\cdot3=12$&$5\cdot17\cdot53=4505$&$54060$&$2919240$\\
$55$&$5\cdot11=55$&$2$&$110$&$6050$\\
\boldmath$56$&\boldmath$2\cdot7=14$&\boldmath$3\cdot13=39$&\boldmath$546$&\boldmath$30576$\\
$57$&$3\cdot19=57$&$2\cdot7=14$&$798$&$45486$\\
$58$&$2\cdot29=58$&$1$&$58$&$3364$\\
$59$&$1$&$1$&$1$&$59$\\
$60$&$4\cdot3\cdot5=60$&$11\cdot19\cdot29\cdot59=357599$&$21455940$&$1287356400$\\
\bottomrule
\end{tabular} \\
\begin{tabular}{ccccc} \toprule
$k$&$a(k)$&$b(k)$&$m(k)/k$&$m(k)$\\ \midrule
$61$&$1$&$1$&$1$&$61$\\
\boldmath$62$&\boldmath$2\cdot31=62$&\boldmath$5\cdot61=305$&\boldmath$18910$&\boldmath$1172420$\\
$63$&$3\cdot7=21$&$2$&$42$&$2646$\\
$64$&$2$&$3\cdot7\cdot31=651$&$1302$&$83328$\\
$65$&$5\cdot13=65$&$2\cdot3=6$&$390$&$25350$\\
$66$&$4\cdot3\cdot11=132$&$5$&$660$&$43560$\\
$67$&$1$&$1$&$1$&$67$\\
$68$&$2\cdot17=34$&$3\cdot67=201$&$6834$&$464712$\\
$69$&$3\cdot23=69$&$2$&$138$&$9522$\\
\boldmath$70$&\boldmath$2\cdot5\cdot7=70$&\boldmath$3\cdot13=39$&\boldmath$2730$&\boldmath$191100$\\
$71$&$1$&$1$&$1$&$71$\\
$72$&$4\cdot3=12$&$5\cdot7\cdot11\cdot17\cdot23\cdot71=10687985$&$128255820$&$9234419040$\\
$73$&$1$&$2$&$2$&$146$\\
$74$&$2\cdot37=74$&$73$&$5402$&$399748$\\
$75$&$3\cdot5=15$&$2$&$30$&$2250$\\
$76$&$2\cdot19=38$&$3\cdot37=111$&$4218$&$320568$\\
$77$&$7\cdot11=77$&$2$&$154$&$11858$\\
$78$&$4\cdot3\cdot13=156$&$5$&$780$&$60840$\\
$79$&$1$&$1$&$1$&$79$\\
$80$&$2\cdot5=10$&$3\cdot7\cdot19\cdot79=31521$&$315210$&$25216800$\\
$81$&$3$&$2$&$6$&$486$\\
$82$&$2\cdot41=82$&$3$&$246$&$20172$\\
$83$&$1$&$1$&$1$&$83$\\
$84$&$4\cdot3\cdot7=84$&$5\cdot11\cdot13\cdot41\cdot83=2433145$&$204384180$&$17168271120$\\
$85$&$5\cdot17=85$&$2$&$170$&$14450$\\
$86$&$2\cdot43=86$&$1$&$86$&$7396$\\
$87$&$3\cdot29=87$&$2$&$174$&$15138$\\
$88$&$2\cdot11=22$&$3\cdot7\cdot43=903$&$19866$&$1748208$\\
$89$&$1$&$1$&$1$&$89$\\
$90$&$4\cdot3\cdot5=60$&$17\cdot29\cdot89=43877$&$2632620$&$236935800$\\
$91$&$7\cdot13=91$&$2\cdot3=6$&$546$&$49686$\\
$92$&$2\cdot23=46$&$3$&$138$&$12696$\\
$93$&$3\cdot31=93$&$2\cdot5=10$&$930$&$86490$\\
$94$&$2\cdot47=94$&$1$&$94$&$8836$\\
$95$&$5\cdot19=95$&$2$&$190$&$18050$\\
$96$&$4\cdot3=12$&$5\cdot7\cdot11\cdot23\cdot31\cdot47=12901735$&$154820820$&$14862798720$\\
$97$&$1$&$1$&$1$&$97$\\
\boldmath$98$&\boldmath$2\cdot7=14$&\boldmath$13\cdot97=1261$&\boldmath$17654$&\boldmath$1730092$\\
$99$&$3\cdot11=33$&$2$&$66$&$6534$\\
$100$&$2\cdot5=10$&$3\cdot7\cdot19=399$&$3990$&$399000$\\
$101$&$1$&$1$&$1$&$101$\\
$102$&$4\cdot3\cdot17=204$&$5\cdot101=505$&$103020$&$10508040$\\
$103$&$1$&$1$&$1$&$103$\\
$104$&$2\cdot13=26$&$3\cdot5\cdot7\cdot103=10815$&$281190$&$29243760$\\
$105$&$3\cdot5\cdot7=105$&$2$&$210$&$22050$\\
$106$&$2\cdot53=106$&$1$&$106$&$11236$\\
$107$&$1$&$1$&$1$&$107$\\
$108$&$4\cdot3=12$&$5\cdot11\cdot17\cdot53\cdot107=5302385$&$63628620$&$6871890960$\\
$109$&$1$&$1$&$1$&$109$\\
$110$&$2\cdot5\cdot11=110$&$3\cdot109=327$&$35970$&$3956700$\\
$111$&$3\cdot37=111$&$2$&$222$&$24642$\\
\boldmath$112$&\boldmath$2\cdot7=14$&\boldmath$3\cdot13=39$&\boldmath$546$&\boldmath$61152$\\
$113$&$1$&$1$&$1$&$113$\\
$114$&$4\cdot3\cdot19=228$&$5\cdot7\cdot37\cdot113=146335$&$33364380$&$3803539320$\\
$115$&$5\cdot23=115$&$2$&$230$&$26450$\\
$116$&$2\cdot29=58$&$3$&$174$&$20184$\\
$117$&$3\cdot13=39$&$2$&$78$&$9126$\\
$118$&$2\cdot59=118$&$1$&$118$&$13924$\\
$119$&$7\cdot17=119$&$2$&$238$&$28322$\\
$120$&$4\cdot3\cdot5=60$&$7\cdot11\cdot19\cdot23\cdot29\cdot59=57573439$&$3454406340$&$414528760800$\\
$121$&$11$&$3$&$33$&$3993$\\
$122$&$2\cdot61=122$&$11$&$1342$&$163724$\\
$123$&$3\cdot41=123$&$2$&$246$&$30258$\\
\boldmath$124$&\boldmath$2\cdot31=62$&\boldmath$3\cdot5\cdot61=915$&\boldmath$56730$&\boldmath$7034520$\\
$125$&$5$&$2$&$10$&$1250$\\
$126$&$4\cdot3\cdot7=84$&$5\cdot13\cdot17\cdot41=45305$&$3805620$&$479508120$\\
$127$&$1$&$2$&$2$&$254$\\
$128$&$2$&$3\cdot7\cdot31\cdot127=82677$&$165354$&$21165312$\\
$129$&$3\cdot43=129$&$2$&$258$&$33282$\\
$130$&$2\cdot5\cdot13=130$&$3$&$390$&$50700$\\
$131$&$1$&$1$&$1$&$131$\\
$132$&$4\cdot3\cdot11=132$&$5\cdot43\cdot131=28165$&$3717780$&$490746960$\\
$133$&$7\cdot19=133$&$2\cdot11=22$&$2926$&$389158$\\
$134$&$2\cdot67=134$&$1$&$134$&$17956$\\
$135$&$3\cdot5=15$&$2$&$30$&$4050$\\
\bottomrule
\end{tabular} \\
\begin{tabular}{ccccc} \toprule
$k$&$a(k)$&$b(k)$&$m(k)/k$&$m(k)$\\ \midrule
$136$&$2\cdot17=34$&$3\cdot7\cdot67=1407$&$47838$&$6505968$\\
$137$&$1$&$1$&$1$&$137$\\
$138$&$4\cdot3\cdot23=276$&$5\cdot137=685$&$189060$&$26090280$\\
$139$&$1$&$1$&$1$&$139$\\
\boldmath$140$&\boldmath$2\cdot5\cdot7=70$&\boldmath$3\cdot13\cdot19\cdot139=102999$&\boldmath$7209930$&\boldmath$1009390200$\\
$141$&$3\cdot47=141$&$2$&$282$&$39762$\\
$142$&$2\cdot71=142$&$1$&$142$&$20164$\\
$143$&$11\cdot13=143$&$3$&$429$&$61347$\\
$144$&$4\cdot3=12$&$5\cdot7\cdot11\cdot17\cdot23\cdot47\cdot71=502335295$&$6028023540$&$868035389760$\\
$145$&$5\cdot29=145$&$2$&$290$&$42050$\\
$146$&$2\cdot73=146$&$1$&$146$&$21316$\\
$147$&$3\cdot7=21$&$2$&$42$&$6174$\\
$148$&$2\cdot37=74$&$3\cdot73=219$&$16206$&$2398488$\\
$149$&$1$&$1$&$1$&$149$\\
$150$&$4\cdot3\cdot5=60$&$7\cdot29\cdot149=30247$&$1814820$&$272223000$\\
\bottomrule
\end{tabular} 
}\end{center}
\section*{Acknowledgments}
This work was supported in part by a Princeton First-Year Fellowship and the NSF Graduate Research Fellowships Program (grant number: DGE-2039656). Initially inspired to think about $m(k)$ as a generalization of the fifth problem on the ELMO 2023 competition,\footnote{See \url{https://web.evanchen.cc/exams/ELMO-2023.pdf}.} the author would like to thank Karthik Vedula for proposing the problem, Daniel Carter for drawing his attention to the reference \cite{ch76}, and Ted Chinburg, Chayim Lowen, Gheehyun Nahm, and an anonymous referee for helpful conversations and/or comments.
\printbibliography
\end{document}